\documentclass[12pt]{amsart}
\copyrightinfo{2012}{Luke Oeding}

\usepackage{amsmath,amsthm,amsfonts,amscd,amssymb,mathrsfs,hyperref}

\usepackage{color}
\newtheorem{theorem}{Theorem}[section]

\newtheorem{lemma}[theorem]{Lemma}

\newtheorem{prop}[theorem]{Proposition}

\newtheorem{cor}[theorem]{Corollary}

\def\rig#1{\smash{ \mathop{\longrightarrow}
 \limits^{#1}}}

\def\O{{\mathcal O}}

\def\M{{\mathcal M}}

\newcommand{\SL}{\operatorname{SL}}
\newcommand{\GL}{\operatorname{GL}}

\newcommand{\V}{{\mathcal V}}

\newcommand{\PP}{\mathbb{P}}
\newcommand{\CC}{\mathbb{C}}
\newcommand{\ZZ}{\mathbb{Z}}

\newcommand{\Seg}{\operatorname{Seg}}

\theoremstyle{definition}

\theoremstyle{remark}
\newtheorem{remark}[theorem]{Remark}

\newcommand{\otdot}{\otimes\cdots\otimes}

\newcommand{\Chow}{\text{Chow}}

\begin{document}

\author{Luke Oeding }
\thanks{This material is based upon work supported by the National Science Foundation under Award No. 0853000: International Research Fellowship Program (IRFP), while the author was in residence primarily at the University of Florence, Italy.  The author also gratefully acknowledges partial support from the Mittag-Leffler institute.}

\address{Dipartimento di Matematica ``U. Dini'' \\
Universit\'a degli Studi di Firenze \\
Firenze, Italy
}
\address{Present Address: Department of Mathematics,
University of California Berkeley,
Berkeley, CA, USA
}
\email{oeding@math.berkeley.edu}

\date{\today}

\title{Hyperdeterminants of polynomials}
\begin{abstract}
The hyperdeterminant of a polynomial (interpreted as a symmetric tensor) factors into several irreducible factors with multiplicities. Using geometric techniques these factors are identified along with their degrees and their multiplicities. The analogous decomposition for the $\mu$-discriminant of polynomial is found.
\end{abstract}

\maketitle

\section{Introduction}
After degree and number of variables, perhaps the most important invariant of a polynomial is the discriminant $\Delta(f)$ - a polynomial in the coefficients of $f$ that vanishes precisely when $f$ has a double root.  Much of the interesting behavior of $f$ is encoded in $\Delta(f)$.

Consider a homogeneous degree $d$ polynomial on $n$ variables $x_{i}$ 
\[
f = \sum_{ 1\leq i_{j}\leq n} a_{i_{1}, \ldots , i_{d}} 
\binom{d}{m_{1}, \ldots , m_{n}}  x_{i_{1}}\cdots x_{i_{d}}
,\]
where $a_{i_{1}, \dots , i_{d}}$ are constants, $m_{j}$ is the number of times that the index $j$ appears in the set $\{i_{1}, \dots ,  i_{d}\}$, and $\binom{d}{m_{1}, \ldots , m_{n}}$ is the multinomial coefficient.
In the case $d=2$, $f$ is equivalent to the matrix of data $A_{f}:=(a_{i,j})_{1\leq i,j\leq n}$, which is symmetric; $a_{j,i} = a_{i,j}$.
% (note that our expression implies a natural choice of basis).
%
It is well known that when $d=2$, the discriminant $\Delta(f)$ is equal to the determinant $\det(A_{f})$.
In general $f$ is equivalent to the $d$-dimensional tensor of data $A_{f}:=(a_{i_{1}, \dots , i_{d}})_{1\leq i_{1}, \dots , i_{d}\leq n}$,  which is (fully) symmetric; $a_{i_{\sigma(1)}, \dots , i_{\sigma(d)}} =  a_{i_{1}, \dots , i_{d}}$ for all permutations $\sigma \in \mathfrak{S}_{d}$.  

A. Cayley \cite{Cayley} introduced the notion of the hyperdeterminant of a multidimensional matrix (tensor) analogous to the determinant of a square matrix.
The hyperdeterminant, whose definition we will recall below, may be thought of in analogy to the discriminant as a polynomial, which tells when a tensor is singular.

The hyperdeterminant went relatively unstudied for approximately 150 years until Gelfand, Kapranov and Zelevisnki brought the hyperdeterminant into a modern light in their groundbreaking work \cite{GKZ92, GKZ}. In particular they determined precisely when the hyperdeterminant is non-trivial and computed the degree.  Inspired by their work, we study the hyperdeterminant applied to a polynomial. We are naturally led to consider the $\mu$-discriminant, which is a partially symmetric analog and generalization of the hyperdeterminant also developed in \cite{GKZ}.
While the hyperdeterminant and $\mu$-discriminant are irreducible, this usually does not continue to hold after symmetrization. Our goal is to determine how the symmetrized hyperdeterminant factors, to determine the geometric meaning of each factor, and to determine the degrees and multiplicities of the factors.  In fact, we will answer these questions for the more general case of the $\mu$-discriminant and the result for the hyperdeterminant will follow as a special case.

The first example that is not a matrix is binary cubics.  The discriminant of a binary cubic has degree $4$.  The hyperdeterminant of a $2 \times 2 \times 2$ tensor also has degree $4$ and the formula is well known (see \cite[(1.5) p.448]{GKZ}). The symmetrization of this polynomial is the discriminant of a binary cubic.  This is the last case that has such simple behavior.

Our curiosity was peaked by the following example that was first pointed out to us by Giorgio Ottaviani.  For plane cubics, the discriminant has degree 12.  The hyperdeterminant of a $3 \times 3 \times 3$ matrix has degree 36. Using Macaulay2 \cite{M2} Ottaviani used Schl\"afli's method to compute the hyperdeterminant, applied this to a symmetric tensor, specialized to a random line and found that the symmetrization of the hyperdeterminant is a reducible polynomial which splits into a factor of degree 12 (the discriminant) and a factor of degree 4 with multiplicity $6$.  The degree $4$ factor turned out to be Aronhold's invariant for plane cubics and defines the variety of Fermat cubics. While Aronhold's invariant is classical, we refer the reader to \cite{Ottaviani09} where one finds a matrix construction which can be applied to construct Aronhold's invariant for degree $3$ symmetric forms on $3$ variables, Toeplitz's invariant \cite{Toeplitz} for triples of symmetric $3 \times 3$ matrices, and  Strassen's invariant \cite{Strassen83} for $3 \times 3 \times 3$ tensors.

After this example, Ottaviani posed the problem to understand and describe this phenomenon in general.
Indeed when $d$ or $n$ are larger than the preceding examples, the hyperdeterminant becomes quite complicated, with much beautiful structure, (see \cite{HSYY, CCDRS}). 
Our approach is to study these algebraic objects from a geometric point of view, thus avoiding some of the computational difficulties, such as those that arise in computing an expansion of the hyperdeterminant in terms of monomials, which would be very difficult in general. 

The outline of the article is the following.
In Section~\ref{sec:combinat} we recall terminology from combinatorics, namely the notion of one partition being refined by another and present a formula for the number of such refinements.
In Section~\ref{sec:multilinear} we review facts from multilinear algebra necessary for our calculations. 
In Section~\ref{sec:geometry} we recall the relevant geometric objects (including Segre-Veronese varieties, Chow varieties and projective duality).   
Finally in Section~\ref{sec:main result} we use geometric methods to prove our main results, which  are the following:

\begin{theorem}\label{thm:hyp}
The $n^{ \times d}$-hyperdeterminant of a symmetric tensor of degree $d\geq 2$ on $n\geq 2$ variables splits as the product\[
\prod_{\lambda} \Xi_{ \lambda, n}^{\M_{\lambda}}
,\]
where $\Xi_{ \lambda, n}$ is the equation of the dual variety of the Chow variety $\Chow_{\lambda}\PP ^{n-1}$ when it is a hypersurface in $\PP^{\binom{n-1+d}{d}-1}$, $\lambda = (\lambda_{1}, \dots , \lambda_{s}) $ is a partition of $d$, and the multiplicity $M_{\lambda} = \binom{d}{\lambda_{1}, \dots , \lambda_{s}}$ is the multinomial coefficient.
\end{theorem}
Geometrically, this theorem is essentially a statement about the symmetrization of the dual variety of the Segre variety. It says that the symmetrization of this dual variety becomes the union of several other varieties (with multiplicities).

In fact, Theorem~\ref{thm:hyp} is a special case of 
 the more general result for Segre-Veronese varieties (see Section~\ref{sec:geometry} for notation):
\begin{theorem}\label{thm:main theorem}
Let $\mu$ be a partition of $d\geq 2$, and $V$ be a complex vector space of dimension $n\geq 2$. Then\[
\Seg_{\mu}\left(\PP V^{ \times t}\right)^{\vee}\cap \PP \left(S^{d}V^{*}\right) = \bigcup_{\lambda \prec \mu} \Chow_{\lambda}\left(\PP V\right)^{\vee}
,\]
where $\lambda\prec \mu$ is the refinement partial order. 
In particular, 
\[\V(Sym(\Delta_{ \mu, n})) = \prod_{\lambda\prec \mu} \Xi_{ \lambda, n}^{M_{\lambda, \mu}}\]
where $\Delta_{ \mu, n}$ is the equation of the hypersurface $\Seg_{\mu}(\PP V^{\times t})^{\vee}$, $\Xi_{ \lambda, n}$ is the equation of $\Chow_{\lambda}\left(\PP V\right)^{\vee}$ when it is a hypersurface in $\PP(S^{d}V)$, and the multiplicity $M_{\lambda,\mu}$ is the number of partitions $\mu$ that refine $\lambda$.
\end{theorem}

We consider only the case where the vector spaces in a tensor product have the same dimension, so \cite[Corollary 3.4]{WeymanZelevinsky_mult} implies that the duals to all Segre-Veronese varieties we will study herein (where the individual factors all have the same dimension) are hypersurfaces.
So we need to know which dual varieties of Chow varieties are hypersurfaces.
\begin{theorem}\label{thm:hypersurface}
Suppose $d\geq 2$, $\dim V =n \geq 2$ and $\lambda =(\lambda_{1}, \dots , \lambda_{s}) =  (1^{m_{1}}, \dots , p^{m_{p}})$ is a partition of $d$. Then $\Chow_{\lambda}\left(\PP V\right)^{\vee}$  a hypersurface with the only exceptions
\begin{itemize}
\item $n = 2$ and $m_{1}\neq 0$ 
\item $n > 2$, $s=2$ and $m_{1}=1$ (so $\lambda =(d-1,1)$).
\end{itemize}
\end{theorem}

In the binary case we have the following closed formula.
\begin{theorem}
The degree of $\Chow_{\lambda}(\PP^{1})^{\vee}$ with $\lambda = (1^{m_{1}},2^{m_{2}}, \ldots , p^{m_{p}})$, $m_{1}=0$ and $m= \sum_{i} m_{i}$ is
\[(m+1)
\binom{m}{m_{2}, \ldots , m_{p}}1^{m_{2}}2^{m_{3}}\cdots (p-1)^{m_{p}} 
\]
\end{theorem}

In more than $2$ variables we have a recursive procedure for computing the degree which is a consequence of Theorem~\ref{thm:main theorem}.
\begin{cor}\label{cor:degree}
Suppose $\dim V \geq 2$.  Let $d_{\lambda}$ denote $ \deg(\Chow_{\lambda}\left(\PP V\right)^{\vee}) $ when it is a hypersurface and $0$ otherwise.
  Then the vector $(d_{\lambda})_{\lambda}$ is the unique solution to the (triangular) system of equations\[
\deg(\Delta_{ \mu, n}) = \sum_{\lambda \prec \mu} d_{\lambda}M_{\lambda,\mu}
.\]
\end{cor}

The multiplicities $M_{\lambda,\mu}$ have a nice generating function.
\begin{prop}\label{mark}
Suppose $\lambda$, $\mu$, are partitions of $d$, $p_{\lambda}$ and $m_{\mu}$ are respectively the power-sum and monomial symmetric functions. Then the matrix $(M_{\lambda,\mu})$ is the change of basis matrix
\begin{equation}
p_{\lambda}(x) = \sum_{\mu\vdash d}M_{\lambda,\mu}m_{\mu}(x)
.\end{equation}
\end{prop}

The degree of $\Delta_{ \mu, n}$ is given by a generating function  (see \cite[Theorem~3.1, Proposition 3.2]{GKZ92} or \cite[page 454]{GKZ}).
So Corollary~\ref{cor:degree} gives a recursive way to compute all of the degrees of the duals of the Chow varieties, and moreover we can package this with Proposition~\ref{mark} into a generating function.

\begin{theorem}\label{thm:degs}
Suppose $\dim V \geq 2$.  Let $d_{\lambda}$ denote $ \deg(\Chow_{\lambda}\left(\PP V\right)^{\vee}) $ when it is a hypersurface and $0$ otherwise.  Let $\Delta_{ \mu, n}$ denote the equation of the hypersurface $\Seg_{\mu}(\PP V^{\times t})^{\vee}$.
The degrees $d_{\lambda}$ are computed by
\[
\sum_{\mu}\deg(\Delta_{ \mu, n})m_{\mu}(x) = \sum_{\lambda } d_{\lambda}p_{\lambda}(x)
,\]
where $m_{\mu}$ and $p_{\lambda}$ are respectively the monomial and power sum symmetric functions.
\end{theorem}

\begin{remark}
The hypersurfaces $\Chow_{\lambda}(\PP V)^{\vee}$ are $\SL(V)$-invariant, and thus each defining polynomial is an $\SL(V)$-invariant for polynomials.  Since invariants of polynomials have been well studied, many of the dual varieties to Chow varieties have alternative descriptions as classically studied objects, however we prefer to ignore these connections for our proofs in order to have a more uniform treatment. However we point out that Corollary~\ref{cor:degree} may be used in retrospect as a way to determine degrees and give geometric interpretations of classical invariants.  In particular, the equations of $\Chow_{\lambda}(\PP V)^{\vee}$ are distinguished $\SL(V)$-invariants in $S^{p}(S^{d}(V))$ (see \cite{Howe}.)
\end{remark}

Recently there has been a considerable amount of work on hyperdeterminants, Chow varieties and related topics, see \cite{GKZ92, Briand, WeymanZelevinsky_sing, WeymanZelevinsky_mult, WeymanBoffi, Carlini, Chipalkatti03,  Chipalkatti04b, LandsbergTensorBook, Hu_Edet}. We are particularly grateful for the very rich book \cite{GKZ}, which provided us both with several useful results and techniques, as well as inspiration. 

In this paper we will work over $\CC$ (or any algebraically closed field of characteristic $0$), it is likely that some of these results can be extended to arbitrary characteristic, but we do not concern ourselves with this problem here. All polynomials will be assumed to be homogeneous.
\section{Combinatorial ingredients}\label{sec:combinat}
An integer vector $\lambda = (\lambda_{1}, \ldots , \lambda_{s})$ is called a 
\emph{partition} of an integer $d$ with $s$ parts if $d\geq \lambda_{1}\geq\dots \geq \lambda_{s}> 0$ and $\sum_{i}\lambda_{i} = d$.  We often shorten this by writing $\lambda \vdash d$ and $\# \lambda = s$.  
The number of repetitions that occur in $\lambda$ may be recorded by writing $\lambda = (1^{m_{1}}, 2^{m_{2}}, \ldots , p^{m_{p}})$, where $i^{m_{i}}$ is to be interpreted as the integer $i$ repeated $m_{i}$ times.

The set of partitions of $d$ is partially ordered by \emph{refinement}. Namely we will write $\lambda\prec\mu$ \footnote{our convention is in the reverse order as in \cite{Stanley}, because we prefer to write $\lambda\prec\mu$ to mimic $\Chow_{\lambda}(\PP V) \subset \Chow_{\mu}(\PP V)$} if the parts of $\mu$ can be partitioned into blocks so that the parts of $\lambda$ are precisely the sum of the elements in each block of $\mu$ \cite[Exercise 3.135]{Stanley}.

Concretely, we will say $\lambda = (\lambda_1, \dots , \lambda_s)$ is \emph{refined} by $\mu$ and write $\lambda \prec \mu$ if there is an expression
\begin{equation}\label{eq:refine}
\begin{matrix}
\lambda_1 = \mu_{i_{1,1}} +\cdots+ \mu_{i_{1,t_1}}\\
\lambda_2 = \mu_{i_{2,1}} +\cdots+ \mu_{i_{2,t_2}}\\
\dots\\
\lambda_s = \mu_{i_{s,1}} +\cdots+ \mu_{i_{t,t_s}}
\end{matrix}
\end{equation}
and 
$\mu= (\mu_{i_{1,1}} , \ldots ,  \mu_{i_{1,t_1}} , \ldots , \mu_{i_{s,1}},  \ldots ,  \mu_{i_{s,t_s}})$ is (after a possible permutation) a partition of $d$.
  Here we emphasize that we do not distinguish two expressions as different if only the orders of the summations in \eqref{eq:refine} change, but we do distinguish the case when different choices of indices of $\mu$ appear in different equations even if some of the $\mu_{i}$ take the same value.

Let $M_{\lambda,\mu}$ denote the number of distinct expressions of the form \eqref{eq:refine} (ignoring rearrangements in the individual summations).
 We will say that  $M_{\lambda,\mu}$ is \emph{number of refinements} from $\mu$ to $\lambda$.\footnote{\cite{GKZ92} uses the same symbol $M_{\lambda,\mu}$ for the Gale-Ryser number, but in \cite{GKZ} they use $d_{\lambda,\mu}$ for the Gale-Ryser number.  We emphasize that our $M_{\lambda,\mu}$ and $d_{\lambda,\mu}$ are related, but not equal.} The refinement partial order is stricter than the dominance partial order.
 
Here are some easy properties of $M_{\lambda,\mu}$ that follow immediately from the definition.
\begin{prop}\label{prop:M properties}
Let $M_{\lambda,\mu}$ denote the number of refinements from $\mu$ to $\lambda$. Then the following properties hold.
\begin{itemize}
\item $M_{(d),\mu} = 1$ for all $|\mu |=d$.
\item $M_{\lambda,\mu}  = 0 $ if $s > t$ or if $s = t$  and $\lambda \neq \mu$, and the matrix $(M_{\lambda,\mu})_{\lambda,\mu}$ is lower triangular for a good choice in ordering of the indices.
\item If $\lambda = (1^{m_{1}}, 2^{m_{2}}, \ldots , p^{m_{p}})$, then $M_{\lambda,\lambda} = m_{1}!\cdots m_{p}!$.
\item $M_{\lambda,1^{d}} = \binom{d}{\lambda}:=\binom{d}{\lambda_{1}, \ldots , \lambda_{s}} = \frac{d!}{\lambda_{1}!\cdots \lambda_{s}!}$, the multinomial coefficient.
\end{itemize}
\end{prop}

One is first tempted to compute $M_{\lambda,\mu}$ via brute force - but this gets complicated quickly.
%For example
%$(3,2)$
%is a refined by
%$(2,1,1,1)$
%and we find that 
%$(3,2) = (2+1,1+1)$  (in 3 different ways)
%and
%$(3,2) = (1+1+1,2)$  (in one way)
%so
%$M_{(2,1,1,1),(3,2)} = 4$.
%
However, one result from the theory of symmetric functions allows for an easy way to compute $M_{\lambda,\mu}$.  Before stating the result, we declare some notation.
 Consider the ring of symmetric functions $\bigwedge[x] = \bigwedge[x_{1},x_{2},\dots]$. For a partition $\lambda = (\lambda_{1},\dots,\lambda_{s})\vdash d$, let $p_{\lambda}\in \bigwedge[x]$ denote the power-sum symmetric function, 
\[
p_{\lambda}(x) = \prod_{i} (x_{1}^{\lambda_{i}}+ x_{2}^{\lambda_{i}}\dots ) 
.\]
For a partition $\mu\vdash d$, let $m_{\mu}\in \bigwedge[x]$ denote the monomial symmetric function,
\[
m_{\mu}(x) = 
\sum_{\sigma \sim} x^{\sigma.\mu}
,\]
where the sum is over distinct permutations $\sigma$ of $\mu = (\mu_{1},\mu_{2},\dots, \mu_{t},0,\dots)$ and $x^{\mu} = x_{1}^{\mu_{1}}\dots x_{t}^{\mu_{t}}\dots$.

Then we have the following (apparently well-known) result. 
\begin{prop}[Proposition~\ref{mark}]
Suppose $\lambda$, $\mu$, $p_{\lambda}$ and $m_{\mu}$ are as above. Then the matrix $(M_{\lambda,\mu})$ is the change of basis matrix
\begin{equation}\label{mlp}
p_{\lambda}(x) = \sum_{\mu\vdash d}M_{\lambda,\mu}m_{\mu}(x)
.\end{equation}
\end{prop}
Thus the matrix $(M_{\lambda,\mu})$ can be quickly computed in any computer algebra system that allows one to compare the coefficients of \eqref{mlp}, namely $M_{\lambda,\mu}$ is the coefficient on the monomial $x^{\mu}$ in \eqref{mlp}.
Note that from Proposition~\ref{mark} also follow the properties listed in Proposition~\ref{prop:M properties}.

\section{Some multi-linear algebra}\label{sec:multilinear}

The elementary facts below will turn out to be useful later. By following the philosophy to not use coordinates unless necessary, we hope to give a more streamlined approach.  As a reference and for much more regarding multilinear algebra and tensors we suggest \cite{LandsbergTensorBook}, which is where we learned this perspective.

Suppose a hyperplane in $\PP V^{\otimes d}$ is represented by a point $[F]$ in $\PP (V^{\otimes d})^{*}$. The multi-linear structure of the underlying vector space $V^{\otimes d}$ allows $F$ to also be considered as a linear map $F\colon V^{\otimes d}\rig{}\CC $, or equivalently as a multilinear form $F\colon V^{ \times d}\rig{}\CC $.  
More explicitly, let $[v_{1}\otimes v_{2}\otimes\cdots\otimes v_{d}] \in \PP\left( V^{\otimes d} \right)$.  Then
\begin{equation}\label{eq:eval}
 F\left(v_{1}\otimes v_{2}\otimes\cdots\otimes v_{d}\right )= F\left(v_{1}, v_{2}, \ldots , v_{d}\right)
,\end{equation}
where on the left we are thinking of $F$ as a linear map, and on the right as a multilinear form.
Our choice of interpretation of $F$ and how to evaluate $F$ will be clear from the context so we will not introduce new notation for the different uses.

Consider $\mu\vdash d$,  $\mu = (\mu_{1}, \ldots ,  \mu_{t}) $ and 
\[u_{1}^{\mu_{1}}\otimes u_{2}^{\mu_{2}}\otimes\cdots\otimes u_{t}^{\mu_{t}} \in S^{\mu_{1}}V \otimes \dots\otimes S^{\mu_{t}}V.\]
The form $F$ may be evaluated on points of $\PP \left(S^{\mu_{1}}V \otimes \dots\otimes S^{\mu_{t}}V \right)$ via the inclusion into $ \PP\left( V^{\otimes d} \right)$
 \[F\left(u_{1}^{\mu_{1}}\otimes u_{2}^{\mu_{2}}\otimes\cdots\otimes u_{t}^{\mu_{t}}\right)
 = F\left(u_{1}, \dots ,  u_{1},u_{2}, \dots , u_{2}, \dots ,  u_{t}, \dots , u_{t}\right)
, \] 
where $u_{i}$ is repeated $\mu_{i}$ times.  

Now suppose $\lambda$ and $\mu$ are such that $M_{\lambda,\mu}$ is non-zero, and consider the inclusion\[
S^{\lambda_{1}}V \otimes S^{\lambda_{2}}V \otimes \cdots \otimes S^{\lambda_{s}}V \subset
S^{\mu_{1}}V \otimes S^{\mu_{2}}V \otimes \cdots \otimes S^{\mu_{t}}V  
,\]
Let $v_{1}^{\lambda_{1}}\otimes \cdots \otimes v_{s}^{\lambda_{s}} \in S^{\lambda_{1}}V \otimes S^{\lambda_{2}}V \otimes \cdots \otimes S^{\lambda_{s}}V$.
Since $v^{j} = v^{\otimes j}$ for any $j$, we may make explicit the above inclusion by writing $v_{1}^{\lambda_{1}}\otimes \cdots \otimes v_{s}^{\lambda_{s}} $
in the form $u_{1}^{\mu_{1}}\otimes u_{2}^{\mu_{2}}\otimes\cdots\otimes u_{t}^{\mu_{t}}$, where each vector $u_{i}$ is an element of $\{v_{1}, \dots , v_{s}\}$ and there is re-ordering of the factors implied by the inclusion above.
 In this case, we say that $u_{1}^{\mu_{1}}\otimes u_{2}^{\mu_{2}}\otimes\cdots\otimes u_{t}^{\mu_{t}}$ \emph{symmetrizes} to $v_{1}^{\lambda_{1}}\otimes \cdots \otimes v_{s}^{\lambda_{s}}$. In addition, there is an inclusion $S^{d} V \subset S^{\lambda_{1}}V \otimes S^{\lambda_{2}}V \otimes \cdots \otimes S^{\lambda_{s}}V $, so we may further symmetrize both points to $v_{1}^{\lambda_{1}} \cdots  v_{s}^{\lambda_{s}}$.

Now suppose $[F]$ is a symmetric hyperplane in $\PP V^{\otimes d}$, $i.e$, $F\in S^{d}V^{*}$. Then \eqref{eq:eval} implies that $F$ \emph{takes the same value at every tensor in $\PP V^{\otimes d}$ that symmetrizes to $v_{1}^{\lambda_{1}} \cdots  v_{s}^{\lambda_{s}}$}.  We will use this fact several times in the sequel. 

As a matter of notation, if $u\in \{v_{1}, \dots , v_{n}\}$ we will write $\frac{v_{1}\cdots v_{n}}{u}$ to denote the product omitting $u$.

\section{Geometric ingredients} \label{sec:geometry}

The hyperdeterminant, the discriminant and their cousins, whose definitions we will recall below, are all equations of irreducible hypersurfaces in projective space, and moreover each hypersurface is the dual variety of another variety.

To say that a polynomial splits into many irreducible factors (with multiplicities) geometrically says that the associated hypersurface decomposes as the union of many hypersurfaces (with multiplicities). 
Geometrically, we would like to describe one dual variety as the union of other dual varieties.
Our perspective is to study the relation between dual varieties and (geometric) symmetrization.  In what follows we will introduce all of the geometric notions we will need to prove our main results.

\subsection{Segre-Veronese and Chow varieties.}

Let $V$ be a complex vector space of dimension $n$. Let $\lambda\vdash d$ with $\# \lambda =s$. Consider the Segre-Veronese embedding via $\mathcal{O}(\lambda)$, which is given by\[
\begin{matrix}
\Seg_{\lambda}:\PP V ^{ \times s} &\rig{|\O(\lambda)|}& \PP\left(S^{\lambda_{1}}V \otimes \dots\otimes S^{\lambda_{s}}V \right) &\subseteq \PP\left( V^{\otimes d}\right) \\
([a_{1}], \dots , [a_{s}]) &\mapsto& [a_{1}^{\lambda_{1}}\otdot a_{s}^{\lambda_{s}}]
.\end{matrix}
\]
We call the image of this map a \emph{Segre-Veronese variety}, and denote it by $\Seg_{\lambda}\left(\PP V^{ \times s}\right)$.  More generally, all of the vector spaces could be different, but we do not need that generality here. 
It is easy to see that $\Seg_{\lambda}\left(\PP V^{ \times s}\right)$ is a smooth, non-degenerate, homogeneous variety of dimension $s(n-1)$.

When $\lambda = (1^{d}) = (1, \dots , 1)$ this is the usual Segre embedding, whose image we will denote by $\Seg\left(\PP V^{ \times d}\right)$, and when $\lambda = (d)$ the map is the $d^{th}$ Veronese embedding, whose image we will denote by $\nu_{d}\left(\PP V\right)$.

Recall that a consequence of the Pieri formula is that for all $\lambda \vdash d$, there is an inclusion\[
S^{d}V \subset S^{\lambda_{1}}V \otimes \dots\otimes S^{\lambda_{s}}V.
\]
Since $G=\GL(V)$ is reductive, there is a unique $G$-invariant complement to $S^{d}V$ in $S^{\lambda_{1}}V \otimes \dots\otimes S^{\lambda_{s}}V$, which we will denote by $W^{\lambda}$.

The linear span of the Segre-Veronese variety is its whole ambient space.  This means, in particular, that there is always a basis of $ S^{\lambda_{1}}V \otimes \dots\otimes S^{\lambda_{s}}V$ consisting of monomials of the form $v_{1}^{\lambda_{1}}\otimes v_{2}^{\lambda_{2}} \otimes\cdots\otimes v_{s}^{\lambda_{s}}$.

For each $\lambda$ there is a natural projection from $W^{\lambda}$, namely \[
\begin{matrix}
\pi_{W^{\lambda}}\colon \PP\left(S^{\lambda_{1}}V \otimes \dots\otimes S^{\lambda_{s}}V \right) &\dashrightarrow& \PP S^{d} V,
\end{matrix}\]
whose definition on decomposable elements is
\[\begin{matrix}
[a_{1}^{\lambda_{1}}\otdot a_{s}^{\lambda_{s}}] &\mapsto & a_{1}^{\lambda_{1}}\cdots a_{s}^{\lambda_{s}}
,\end{matrix}
\]
and is extended by linearity.

For each $\lambda$ we define a \emph{Chow variety}, denoted $\Chow_{\lambda}(\PP V)$, as the image of the Segre-Veronese variety under the projection $\pi_{W^{\lambda}}$.  The image of the projection is not changed by permutations acting on $\lambda$. So $\Chow_{\lambda}(\PP V^{ \times s})$ is equally the projection of $\Seg_{\sigma(\lambda)}(\PP V^{ \times s})$ for any permutation $\sigma\in \mathfrak{S}_{s}$, i.e. $\Chow_{\lambda}(\PP V^{ \times s}) =\Chow_{\sigma(\lambda)}(\PP V^{ \times s})$.  The number of unique projections is $M_{\lambda,\lambda}$.  On the other hand, $\Seg_{\lambda}(\PP V^{ \times s})$ and $\Seg_{\pi}(\PP V^{ \times s})$ are (only) isomorphic if $\pi =\sigma(\lambda)$ for some permutation $\sigma$, and \emph{equal} only if $\pi=\lambda$.

Chow varieties are also sometimes called coincident root loci (see \cite{Chipalkatti04b} related to the case $n=2$). When $\lambda = (1^{d})$, the Chow variety is the variety of polynomials that are completely reducible as a product of linear forms, and is sometimes called the split variety, \cite{ArrondoBernardi}. 
 For general $\lambda$, the Chow variety is the closure of the set of polynomials that are completely reducible as the product of linear forms that are respectively raised to powers $\lambda_{1}, \ldots , \lambda_{s}$. One can check that the definition we have given is equivalent to the usual definition of a Chow variety, see \cite{Carlini}. 

The following is well known (see \cite{Chipalkatti04b} for example).
\begin{prop}
$\dim(\Chow_{\lambda}(\PP V)) = (\#\lambda) (n-1)$.
\end{prop}
\begin{proof} Let $\dim(V) = n$, and $d = |\lambda|$.
The Segre-Veronese map $\PP  V  \times\dots \times  \PP V \rig{} \PP \left( S^{\lambda_{1}}V \otimes\cdots\otimes S^{\lambda_{s}}V \right)$ is an embedding, and in particular the dimension of the image is $s(n-1)$.  The projection to $\PP S^{d}V$ is a finite morphism, so the image is also $s(n-1)$-dimensional.
\end{proof}

\begin{remark}
It is interesting to note that the refinement partial order on partitions also exactly controls the containment partial order on Chow varieties.  Namely\[
\Chow_{\lambda}(\PP V) \subset \Chow_{\mu}(\PP V)
\]
precisely when $\lambda \prec \mu$.

We also note that $\Seg \PP V^{\times d} \cap \PP S^{d} V = \nu_{d}(\PP V)$.
More generally, if $\lambda\prec \mu$, then $\Seg_{\mu} \PP V^{\times t} \cap \PP (S^{\lambda_{1}} V \otimes \cdots \otimes S^{\lambda_{s}}V ) = \Seg_{\lambda}\PP V$ (after appropriately re-ordering).
\end{remark}

\subsection{Dual varieties}Let $U$ denote a complex, finite dimensional vector space and let $U^{*}$ denote the dual vector space of linear forms $\{U\rig{}\CC\}$.
For a smooth projective variety $X \subset \PP U$, the dual variety $X^{\vee} \subset \PP U^{*}$ is the variety of tangent hyperplanes to $X$.  Specifically, let $\widehat T_{x}X \subset U$ denote the cone over the tangent space to $X$ at $[x]\in X$. 
The \emph{dual variety} of $X$ in $\PP U^{*}$ is defined as\[
X^{\vee}:= 
\left\{ 
[H] \in \PP U^{*} \mid \exists [x] \in X, \widehat T_{x}X \subset H
 \right\}
.\]
\begin{remark}
If $X\subset \PP U$ is not smooth the dual variety can still be defined with a bit more care.
Consider the incidence variety (conormal variety)\[
\mathcal{P} =
\{
 ([x],[H]) \mid \widehat{T}_{x} \subset H
\}
\subset \PP U  \times \PP U^{*} \subset \PP (U\otimes U^{*})
,\]
which we define only for smooth points of $X$ and then take the Zariski closure (see \cite{Zak} for a more thorough treatment).
The conormal variety is equipped with projections $p_{1}$ and $p_{2}$ to the first and second factors respectively. The projection $p_{2}$ to the second factor defines $X^{\vee}$.
\end{remark}
Recall that the dual variety of an irreducible variety is also irreducible, \cite[Proposition~1.3~p.15]{GKZ}.
Usually, we expect the dual variety $X^{\vee}$ to be a hypersurface.  When this does not occur, we say that $X$ is \emph{defective}.

The dual variety of the Veronese $\nu_{d}\left(\PP V\right)^{\vee}$ is a hypersurface defined by the classical discriminant of a degree $d$ polynomial on $n$ variables, which we will denote $\Delta_{(d),n}$, see \cite[Example~I.4.15, p.38]{GKZ}.  We are told in the same passage that G. Boole in 1842 introduced this discriminant and found that $\deg(\Delta_{(d),n})  = (n)(d-1)^{n-1}$.

The \emph{hyperdeterminant} of format $n^{ \times d}$, denoted $HD_{n,d}$,
is the equation of the (irreducible) hypersurface 
$\Seg\left(\PP V^{ \times d}\right)^{\vee} \subset \PP \left(V^{\otimes d}\right)^{*}$.  Note that $HD_{n,d}$ is a polynomial of degree $N(n,d)$ on $\left(V^{\otimes d}\right)^{*}$ where $N(n,d)$ can be computed via the generating functions found in \cite[Theorem~3.1, Proposition~3.2]{GKZ92} or also \cite[Theorem~XIV.2.4, p~454]{GKZ}.  

Segre-Veronese varieties and their duals are also well-studied objects. In particular, it is known precisely when they are hypersurfaces \cite[Proposition XIII.2.3~p.441]{GKZ}, and their degree is given in \cite[Theorem~XIII.2.4~p.441]{GKZ} via a nice generating function.  Since we consider multiple copies of $V$ rather than vector spaces of different dimensions the dual of Segre-Veronese varieties are always hypersurfaces. For $\mu\vdash d$ we denote by $\Delta_{ \mu, n}$  the \emph{$\mu$-discriminant}\footnote{The $\mu$-discriminant is called an $A$-discriminant in \cite{GKZ}.}, which is the equation of $\Seg_{\mu}\left(\PP V^{ \times t}\right)^{\vee}$.

When the dual variety $\Chow_{\lambda}\left(\PP V\right)^{\vee}$ is a hypersurface (see Theorem~\ref{thm:hypersurface}),  we will let $\Xi_{ \lambda, n}$ denote its equation, which is unique up to multiplication by a non-zero scalar.

\subsection{Projections and dual varieties}
The focus of this article is the symmetrization of the hyperdeterminant.
In general \emph{the symmetrization of a polynomial} (whose variables are elements of a tensor product) is the map induced by the map that symmetrizes the variables. This may be described invariantly as follows.  If $f\in S^{e}(V^{\otimes d})^{*}$ is a degree $e$ homogeneous polynomial on $V^{\otimes d}$, then $Sym(f)$ is the image of $f$ under the projection $S^{e}(V^{\otimes d})^{*}\rig{} S^{e}(S^{d}V)^{*}$.  While this map can be described in bases in complete detail, we do not need this for the current work.

To study the dual varieties of the varieties we have introduced we need to understand relation between taking dual variety and taking projection.
This is the content of the following proposition, which can be found in Landsberg's book, \cite[Proposition~8.2.6.1]{LandsbergTensorBook}, and is similar to \cite[Prop~4.1~p.31]{GKZ} and closely related to \cite{Holme88}.  If $W$ is a subspace of $V$ let $W^{\perp}$ denote the annihilator of $W$ in the dual vector space $V^{*}$, which is isomorphic to the quotient $(V/W)^{*}$.
\begin{prop}[\cite{LandsbergTensorBook}]\label{prop:project}
Let $X \subset \PP V$ be a variety and let  $W \subset V$ be a linear subspace. 
Consider the rational map $\pi\colon \PP V \dashrightarrow \PP (V/W)$.
Assume $X \not\subset \PP W$. Then\[
\pi(X)^{\vee}\subseteq  X^{\vee} \cap \PP W^{\perp}
\]
and if $\pi(X)\cong X$, then equality holds.
\end{prop}
When $W = W^{\lambda} = (S^{d}V^{*})^{\perp} \subset S^{\lambda_{1}}V \otimes\cdots \otimes S^{\lambda_{s}}V$, the map $\pi$ is symmetrization and we will denote it by $Sym$.

\begin{prop}\label{prop:intersect}
Suppose $U$ is a subspace of $V$.  Let $X\subset \PP V$ be a variety and let $Y =  X \cap \PP U$.  Then $Y^{\vee} \subset X^{\vee}\cap \PP (V/U)^{\perp}$.
\end{prop}
We identify $\PP (V/U)^{\perp}$ with $U^{*}$ and by abuse of notation write $Y^{\vee} \subset X^{\vee}\cap \PP U^{*}$.
\begin{proof}
This statement is almost a tautology, but we give a proof anyway.
We prove this for smooth points of $Y$, the rest follows by standard arguments taking Zariski closure.

Let $H^{U}$ be a hyperplane in $U$ tangent to $Y$ at some point $y$.  Then $H^{U}$ is naturally associated to a point $p \in U^{*}$, and we identify $U^{*}\cong (V/U)^{\perp}$.  This identification allows us to conclude that $p$ (now considered as a point in $V^{*}$) is associated to a hyperplane $H$ in $V$ that contains $V/U$. We will write this as $H = H^{U}+V/U$.  To conclude we must show that $H$ is tangent to $X$ at some point.

The direct sum $V = U \oplus W$, where $W = V/U$ induces a decomposition of $X$ and of the tangent space of $X$.  In particular, we may assume at a general point $p\in X$ that there is $w\in X\cap \PP W$, so that $p=y+w \in X$ and the decomposition is
\[
T_{y+w}X = T_{y}Y + T_{w}(X\cap \PP W) 
.\]
But from this description we see that since $H^{U}\supset T_{y}Y$, and by definition  $W \supset T_{w}(X\cap \PP W)$, so $H$ is tangent to $X$ at $y+w$.
\end{proof}

The following statement, which follows directly from the definition, relates the symmetrization of the $\mu$-discriminant to the geometric setting. (This statement is essentially \cite[Cor. 4.5]{GKZ}.)

\begin{prop}[Proposition/Definition]\label{prop:int} Let $\mu\vdash d$.
The symmetrization of the $\mu$-discriminant is the $\mu$-discriminant of a  symmetric tensor\[
\V(Sym(\Delta_{ \mu, n})) = \Seg_{\mu} \left(\PP V^{ \times t}\right)^{\vee} \cap \PP \left(S^{d}V^{*} \right)
,\]
where we are using the isomorphism $S^{d}V^{*}\cong \left( (S^{\mu_{1}}V\otimes \cdots \otimes S^{\mu_{t}}V)/S^{d}V \right)^{\perp}$.
\end{prop}
In particular, the symmetrization of the hyperdeterminant (of format $n^{\times d}$) is the hyperdeterminant of a symmetric multi-linear form;
 \[
 \V(Sym(HD_{d,n})) = \Seg\left(\PP V^{ \times d}\right)^{\vee}\cap \PP \left(S^{d}V^{*}\right)
.\]

\subsection{Plane cubics again}
As a prototypical example, we return to plane cubics. It was classically known that decomposable plane cubics and Fermat cubics are related by projective duality.
\begin{prop}\label{prop:easy}  Consider $\Chow_{1,1,1}\PP^{2}= \{[l_{1}l_{2}l_{3}] \in \PP S^{3} \CC^{3} \mid  0\neq l_{i}\in \CC^{3} \} \subset \PP S^{3}\CC^{3}$.
 $\Chow_{1,1,1}\left(\PP V\right)^{\vee}$ is the closure of the orbit of the Fermat cubic, \emph{i.e.} the $3^{rd}$ secant variety to the cubic Veronese:\[
\Chow_{1,1,1}\left(\PP^{2}\right)^{\vee} = \sigma_{3}( \nu_{3}\PP^{2}).
\]
\[= \overline{ \left\{ 
h\in \PP S^{3}(\CC^{3})^{*}\mid h = e_{1}^{3} + e_{2}^{3}+ e_{3}^{3}, e_{i}\in (\CC^{3})^{*}
\right\}} 
\subset \PP S^{3}(\CC^{3})^{*}.\]
In particular, $\Chow_{1,1,1}\left(\PP V\right)^{\vee}$ is a hypersurface.
\end{prop}

The proof of~\ref{prop:easy} is a straightforward calculation considering the conditions imposed on a hyperplane in $S^{3}V^{*}$ that annihilates a tangent vector through a curve on the Chow variety of the form $l_{1}(t)l_{2}(t)l_{3}(t)$, where for each $t$ $l_{i}(t)$ is a linear form. We leave the details for the reader.

Lemma~\ref{lem:part 1} below implies that since $\Chow_{1,1,1}\left(\PP V\right)^{\vee}$ is a hypersurface, its equation must divide the symmetrization of the hyperdeterminant of format $3 \times 3 \times 3$.
This geometric statement, however, ignores multiplicity.
Because of our generality assumption, there are six different tensors -- $l_{1}\otimes l_{2}\otimes l_{3}$ and its permutations -- that symmetrize to $l_{1}l_{2}l_{3}$.
This fact implies that there are $6$ copies of the equation of $\Chow_{1,1,1}\left(\PP V\right)^{\vee}$ in the symmetrized hyperdeterminant.

Proposition~\ref{prop:easy} is characteristic of the theme of the rest of the article.  The splitting of the hyperdeterminant of a polynomial will depend on the dimensions and multiplicities of the dual varieties of Chow varieties. We also will show that it this is sufficient.

\begin{remark}
One may attempt to do something similar to Proposition~\ref{prop:easy} more generally for $\Chow_{\lambda}(\PP V)$ for any $\lambda \vdash d$, and $\#\lambda = s\leq n$.
One finds that (as long as $s\leq n$),\[
\Chow_{\lambda}(\PP V)^{\vee} \supset \sigma_{d}(\nu_{d}(\PP V)),
\]
where $\sigma_{d}(\nu_{d}(\PP V))$ is the variety of points on secant $d-1$-planes to the Veronese variety $\nu_{d}(\PP V)$.
Equality does not hold in general. One may use the dual of Chow varieties as a source for equations for secant varieties of Veronese varieties.  The utility of this fact is limited by the degree of the equations obtained.
\end{remark}

\subsection{Dimension of the duals to Chow varieties}

Herein we prove Theorem~\ref{thm:hypersurface} about the dimension of the duals of Chow varieties. For the reader's convenience, we repeat
that we need to show that for $\lambda = (1^{m_{1}},  \dots ,  p^{m_{p}})$, and $n=\dim(V)$,  $\Chow_{\lambda}\left(\PP V\right)^{\vee}$  a hypersurface with the only exceptions
\begin{itemize}
\item $n = 2$ and $m_{1}\neq 0$ 
\item $n > 2$, $s=2$ and $m_{1}=1$ (so $\lambda =(d-1,1)$).
\end{itemize}
The case $d=2$ is already well understood, so we will assume $d>2$.

\begin{proof}[Proof of Theorem~\ref{thm:hypersurface}]
The dimension of a dual variety can be calculated via the Katz dimension formula, essentially calculating the Hessian at a general point, but we prefer to work geometrically.

A dual variety $X^{\vee}$ is a hypersurface unless a general tangent hyperplane is tangent to $X$ in a positive dimensional space. This is the condition that we will apply in both cases.  Our proof follows a standard proof about the non-degeneracy for the dual of Segre-Veronese varieties.

For any $n$, the Chow variety $\Chow_{\lambda}(\PP V)$ does not contain any linear spaces if $m_{1}=0$, so in this case the dual is a hypersurface. 

Now suppose $m_{1}>0$.  We must then show that a generic hyperplane is tangent to $\Chow_{\lambda}(\PP V)$ along (at least) a line precisely when $n=2$ or when $n>2$ and $\lambda = (d-1,1)$. 

Consider a general point $[x] \in\Chow_{\lambda}(\PP V)$ with $\lambda = (1^{m_{1}}, \ldots ,  p^{m_{p}})$ and $m_{1}>0$.  Then we may write  $x = lf$, where $f$ is completely decomposable and $l$ is a linear form.
Then the tangent space is\[
\begin{matrix}
T_{x}\Chow_{\lambda}(\PP V) &=& \{lf,wf, lf'z\mid w,z \in V \} \\
&=& V\cdot f+\sum_{i} V\cdot l\cdot \frac{f}{y_{i}}  
,\end{matrix}
\]
where $y_{i}$ are the factors of $f$.

The linear space $\PP L = \PP (V\cdot f)$ is contained in $\Chow_{\lambda}(\PP V)$, and up to reordering of the factors of $x$, every linear space on $\Chow_{\lambda}(\PP V)$ is of this form.

Suppose $H$ is a general hyperplane that contains a general tangent space $T_{x}:=T_{x}\Chow_{\lambda}(\PP V)$ as above. If $H$ is to be tangent along a line on $\Chow_{\lambda}(\PP V)$, then there must be another point of the form $[l'f] \in V\cdot f$ that is distinct from $[x]=[lf]$.  So we must choose 
 $l' \in l^{\perp}\subset V$ and $l^{\perp}$ is an $n-1$ dimensional vector space.
Next we consider two cases, first $\#\lambda = s=2$ and later $s>2$.

If $s=2$, consider $x = l_{1}\cdot f$, with $f=l_{2}^{d-1}$, and generically we may assume $l_{1}$ and $l_{2}$ are linearly independent.
Let $y=l_{3}l_{2}^{d-1}$ be a general point in $Vl_{2}^{d-1}$, where $l_{3}$ is assumed to be independent of $l_{1}$ so that $x$ and $y$ are independent. 

Since $H$ annihilates $T_{x}$, we should calculate $T_{y}$ modulo $T_{x}$.  The vectors that remain are all of the form  $l_{3}l_{2}'l_{2}^{d-2}$.  If $l_{2}'$ is in the line $[l_{2}]$ then $l_{3}l_{2}'l_{2}^{d-2}$ is contained in $T_{x}$.  Additionally, if $l_{2}'$ is in the line $[l_{1}]$, then $l_{2}l_{3}'l_{3}^{d-2}$ is not on $\Chow(1,d-1)\PP V$.  So we generically have a non-trivial condition $H(l_{2}l_{3}'l_{3}^{d-2})=0$ for each $l_{2}' \in \{l_{1},l_{2} \}^{\perp}$, which is at most $n-2$ conditions.  Therefore the dimension of the space of possible points $l'f$ is at least $n-1 -(n-2) = 1$, thus a generic hyperplane is tangent along a line. 

Now suppose $s>2$.  We will consider first the case $s=3$ and later argue that considering this case suffices.

Let $x = l_{1}l_{2}^{i}l_{3}^{j}$, where $i+j+1=d$ and $i,j>0$, else we revert to the previous case.  Consider $y = l_{4}l_{2}^{i}l_{3}^{j} \in Vl_{2}^{i}l_{3}^{j}$ and compute $T_{y}$ modulo $T_{x}$.  Points on $T_{y}$ have the form\[
l_{4}'l_{2}^{i}l_{3}^{j} + i\cdot l_{4}l_{2}'l_{2}^{i-1}l_{3}^{j} +j\cdot l_{4}l_{2}^{i}l_{3}l_{3}^{j-1}
,\]
which reduces to \[
i\cdot l_{4}l_{2}'l_{2}^{i-1}l_{3}^{j} +j\cdot l_{4}l_{2}^{i}l_{3}l_{3}^{j-1}
,\]
modulo $T_{x}$.
As before, the maximum number of independent conditions we can impose on the choices of $y\in Vl_{2}^{i}l_{3}^{j}$ will come from the cases when\[
l_{2}'\in \{l_{1},l_{2}\}^{\perp} \text{ and } l_{3}' \in \{l_{1},l_{3}\}^{\perp}
,\]
which are $n-2 +n-2=2n-4$ conditions,
and for generic $H$ this bound will be achieved.  When $n=2$ no additional conditions are imposed and $\Chow_{i,j,1}\PP^{1}$ is not a hypersurface. 
On the other hand, $2n-4$ independent conditions imposed on a space of dimension $n-1$ will not have positive dimension as soon as $n\geq 3$, and thus $\Chow_{i,j,1} \PP V$ is a hypersurface whenever $\dim V\geq 3$.

Finally, when $s>3$ the analogous calculation provides at least as many conditions to impose on the $n-1$ choices of possible additional points in $Vf$ where a generic hyperplane may be tangent to $\Chow_{\lambda}(\PP V)$, so the dimension of the resulting space will not be positive for $\dim V\geq 3$.
\end{proof}

\section{Proof of main results}\label{sec:main result}
We will prove Theorem~\ref{thm:main theorem} in two steps. The first step is the following.

\begin{lemma}\label{lem:part 1}
Suppose $\lambda\vdash d$ with $\#\lambda =s$. Then for every $\mu\vdash d$ with $\# \mu = t$  such that $\lambda \prec \mu$ (by refinement)
\begin{equation} \label{eq:chowseg}
\Chow_{\lambda}\left(\PP V\right)^{\vee}\subset \Seg_{\mu}\left(\PP V^{ \times t}\right)^{\vee}\cap \PP \left(S^{d}V^{*}\right)
.\end{equation}
Moreover when $\Chow_{\lambda}\left(\PP V\right)^{\vee}$ is a hypersurface  it occurs with multiplicity $M_{\lambda,\mu}$ in $\Seg_{\mu}\left(\PP V^{ \times t}\right)^{\vee}\cap \PP \left(S^{d}V^{*}\right)
$, where $M_{\lambda,\mu}$
 is the number of partitions $\mu$ that refine $\lambda$.
\end{lemma}

A generating function for $M_{\lambda,\mu}$ is given in Proposition~\ref{mark}.
We will give two proofs of Lemma~\ref{lem:part 1}.  The first relies only on the two statements Propositions~\ref{prop:project} and~\ref{prop:intersect} and is more efficient, however it only proves a lower bound for the multiplicity.  The second is more computational and gives some ideas as to how the subsequent statements will be proved. The second proof has the advantage of providing an exact count for the multiplicity.

\begin{proof}[Proof 1]
Suppose as in the statement that $\lambda \prec \mu$.  Proposition~\ref{prop:project} implies
\[
\Chow_{\lambda}(\PP V )^{\vee} 
\subset
\Seg_{\lambda}(\PP V ^{\times t} )^{\vee} \cap \PP( S^{d}V)^{*}
.\]
Up to re-ordering of the tensor product we have
\[
\Seg_{\lambda} (\PP V^{\times s})= \Seg_{\mu}(\PP V ^{\times t} ) \cap \PP (S^{\lambda_{1}}\otimes\cdots\otimes S^{\lambda_{s}}V)
,\]
so 
\[
\Chow_{\lambda}(\PP V )^{\vee} 
\subset
\left(\Seg_{\mu}(\PP V ^{\times t} ) \cap \PP (S^{\lambda_{1}}\otimes\cdots\otimes S^{\lambda_{s}}V) \right)^{\vee} \cap \PP S^{d}V^{*} 
.\]
Proposition~\ref{prop:intersect} implies that the right hand side of the above expression satisfies
\begin{align*}
&\left(\Seg_{\mu}(\PP V ^{\times t}  \cap  \PP (S^{\lambda_{1}}\otimes\cdots\otimes S^{\lambda_{s}}V) \right)^{\vee}
\cap \PP( S^{d}V)^{*} \\
&\subset
\left(\Seg_{\mu}(\PP V ^{\times t} )^{\vee} \cap  \PP (S^{\lambda_{1}}\otimes\cdots\otimes S^{\lambda_{s}}V)^{*} \right)
\cap \PP( S^{d}V)^{*} \\
&=
\Seg_{\mu}(\PP V ^{\times t} )^{\vee}
\cap \PP( S^{d}V)^{*} 
,\end{align*}
and this yields the result \eqref{eq:chowseg}.

Finally, this occurs for every $\lambda$ and $\mu$ for which we have $\lambda\prec \mu$, so $M_{\lambda,\mu}$ is a lower bound for the multiplicity of $\Chow_{\lambda}\subset  \Seg_{\mu}(\PP V ^{\times t} )^{\vee} \cap \PP( S^{d}V)^{*} $.
\end{proof}
\begin{proof}[Proof 2]
Suppose $F$ is a symmetric hyperplane tangent to $\Chow_{\lambda}\left(\PP V^{ \times s}\right)$ at a general point $[v_{1}^{\lambda_{1}}\cdots v_{s}^{\lambda_{s}}]$.  Then we have 
\begin{equation}\label{eq:ltan}
F\left(w\frac{v_{1}^{\lambda_{1}}\cdots v_{s}^{\lambda_{s}}}{v_{i}}\right)=0
\end{equation}
for all $1\leq i \leq s$ and for all $w \in V$.  

Now we apply the ideas outlined in Section~\ref{sec:multilinear}.
Since $\lambda \prec \mu$,  we can consider the inclusion 
\begin{equation}\label{eq:partial}
S^{\lambda_{1}}V \otimes S^{\lambda_{2}}V \otimes \cdots \otimes S^{\lambda_{s}}V 
\subset
S^{\mu_{1}}V \otimes S^{\mu_{2}}V \otimes \cdots \otimes S^{\mu_{t}}V  
.\end{equation}

This implies that $v_{1}^{\lambda_{1}}\cdots v_{s}^{\lambda_{s}} \in S^{\mu_{1}}V\otimes\cdots\otimes S^{\mu_{t}} V$, and there exists a tensor $u_{1}^{\mu_{1}}\otimes u_{2}^{\mu_{2}}\otimes\cdots\otimes u_{t}^{\mu_{t}}$, where each vector $u_{i}$ is an element of $\{v_{1}, \dots , v_{s}\}$, and in particular $u_{1}^{\mu_{1}}u_{2}^{\mu_{2}}\cdots u_{t}^{\mu_{t}} = v_{1}^{\lambda_{1}}\cdots v_{s}^{\lambda_{s}} $.
In other words $u_{1}^{\mu_{1}}\otimes u_{2}^{\mu_{2}}\otimes\cdots\otimes u_{t}^{\mu_{t}}$ symmetrizes to $v_{1}^{\lambda_{1}}\cdots v_{s}^{\lambda_{s}} $, and thus $F\left(u_{1}^{\mu_{1}}\otimes u_{2}^{\mu_{2}}\otimes\cdots\otimes u_{t}^{\mu_{t}}\right)=0$ (see Section~\ref{sec:multilinear}). 

We claim that $F$ is tangent to $\Seg_{\mu}\left(\PP V^{ \times t}\right)$ at each $[u_{1}^{\mu_{1}}\otimes u_{2}^{\mu_{2}}\otimes\cdots\otimes u_{t}^{\mu_{t}}]$.
Indeed, any tangent vector through $u_{1}^{\mu_{1}}\otimes u_{2}^{\mu_{2}}\otimes\cdots\otimes u_{t}^{\mu_{t}}$ can be written as a linear combination of tensors of the form
 \[
 u_{1}^{\mu_{1}}\otimes\cdots\otimes u_{i-1}^{\mu_{i-1}}\otimes w\cdot u_{i}^{\mu_{i}-1}\otimes u_{i+1}^{\mu_{i+1}}\otimes\cdots\otimes  u_{t}^{\mu_{t}}
 \]
 for $1\leq i\leq t $ and $w\in V$.  This tensor symmetrizes to 
 \[w\cdot\frac{
 u_{1}^{\mu_{1}}\cdots  u_{t}^{\mu_{t}}}{u_{i}}
= w\cdot\frac{
 v_{1}^{\lambda_{1}}\cdots v_{s}^{\lambda_{s}} 
 }{v_{j}}\]
where the equality holds because 
$u_{1}^{\mu_{1}}\otimes u_{2}^{\mu_{2}}\otimes\cdots\otimes u_{t}^{\mu_{t}}$ 
symmetrizes to
 $ v_{1}^{\lambda_{1}}\cdots v_{s}^{\lambda_{s}} $
 and moreover $u_{i}=v_{j}$ for some $j$.  Since $F$ is symmetric and takes the same value at every tensor that symmetrize to the same form,
 \eqref{eq:ltan} implies that $F$ is tangent to  $\Seg_{\mu}\left(\PP V^{ \times t}\right)$ at each $[u_{1}^{\mu_{1}}\otimes u_{2}^{\mu_{2}}\otimes\cdots\otimes u_{t}^{\mu_{t}}]$.

The number of points $u_{1}^{\mu_{1}}\otimes u_{2}^{\mu_{2}}\otimes\cdots\otimes u_{t}^{\mu_{t}} \in S^{\mu_{1}}V\otimes\cdots\otimes S^{\mu_{t}} V$ that symmetrize to $v_{1}^{\lambda_{1}}\cdots v_{s}^{\lambda_{s}}$ is computed by $M_{\lambda,\mu}$ and is a lower bound for the multiplicity of $\Chow_{\lambda}$ in $\Seg_{\mu}\left(\PP V^{ \times t}\right) \cap \PP S^{d}V^{*}$.

On the other hand, suppose $\Chow_{\lambda}(\PP V)^{\vee}$ is a hypersurface and is contained in $\Seg_{\mu}\left(\PP V^{ \times t}\right)\cap \PP S^{d}V^{*}$. 

Suppose $u_{1}^{\mu_{1}}\otimes u_{2}^{\mu_{2}}\otimes\cdots\otimes u_{t}^{\mu_{t}} \in S^{\mu_{1}}V\otimes\cdots\otimes S^{\mu_{t}} V$ is a tensor that does not symmetrize to $v_{1}^{\lambda_{1}}\cdots v_{s}^{\lambda_{s}}$ but still $[u_{1}^{\mu_{1}}\cdots u_{t}^{\mu_{t}}] \in \Chow_{\lambda}\PP V$.   In particular 
$[u_{1}^{\mu_{1}}\cdots u_{t}^{\mu_{t}}]
\neq 
[v_{1}^{\lambda_{1}}\cdots v_{s}^{\lambda_{s}}]$.

If $[u_{1}^{\mu_{1}}\cdots u_{t}^{\mu_{t}}] \in T_{[v_{1}^{\lambda_{1}}\cdots v_{s}^{\lambda_{s}}]}\Chow_{\lambda}(\PP V) \subset [F]$, then $F$ is not tangent to $\Chow_{\lambda}(\PP V)$  at $[u_{1}^{\mu_{1}}\cdots u_{t}^{\mu_{t}}]$ else this would violate the condition that $\Chow_{\lambda}\PP V^{\times s}$ be a hypersurface.

If $[u_{1}^{\mu_{1}}\cdots u_{t}^{\mu_{t}}]$ is not in the tangent space and $\Chow(\PP V)$ is not the whole ambient space, a generic $F$ satisfying $s(n-1)$ independent conditions will miss a point, thus we can choose an $F$ that does not vanish at $[u_{1}^{\mu_{1}}\cdots u_{t}^{\mu_{t}}]$.

So $M_{\lambda,\mu}$ is also the maximum multiplicity of a hypersurface $\Chow_{\lambda}(\PP V)^{\vee}$ in $\Seg_{\mu}(\PP V^{ \times t})^{\vee}\cap \PP S^{d}V^{*}$.
\end{proof}

The second step of the proof of Theorem~\ref{thm:main theorem} is the following.
\begin{lemma}\label{lem:part 2}
Suppose $F\subset V^{\otimes d}$ is a symmetric hyperplane that is tangent to the Segre-Veronese variety $\Seg_{\mu}\left(\PP V^{ \times t}\right)$ at $[u_{1}^{\otimes \mu_{1}}\otimes\cdots\otimes u_{t}^{\otimes \mu_{t}}]$.  
Suppose $\lambda\prec \mu$. Then $[u_{1}^{\mu_{1}}\cdots u_{t}^{\mu_{t}}] \in \Chow_{\lambda}\left(\PP V\right)$ and $F$ is also tangent to $\Chow_{\lambda}(\PP V)$ at $[u_{1}^{\mu_{1}}\cdots u_{t}^{\mu_{t}}]$.
\end{lemma}
\begin{proof}
By hypothesis since $\lambda \prec \mu$, there is a symmetrization of $u_{1}^{\otimes \mu_{1}}\otimes\cdots\otimes u_{t}^{\otimes \mu_{t}}$ so that
$u_{1}^{\mu_{1}}\cdots u_{t}^{\mu_{t}} = 
v_{1}^{\lambda_{1}}\cdots v_{s}^{\lambda_{s}}$
and for all $1\leq i\leq s$, $v_{i}\in \{u_{1}, \dots , u_{t}\}$.

The conditions that $F$ be tangent to  $\Seg_{\mu}\left(\PP V^{ \times t}\right)$ at $[u_{1}^{\otimes \mu_{1}}\otimes\cdots\otimes u_{t}^{\otimes \mu_{t}}]$ are
\[\begin{matrix}
F\left( u_{1}^{\mu_{1}}\otimes\cdots\otimes u_{i-1}^{\mu_{i-1}}\otimes w\cdot u_{i}^{\mu_{i}-1}\otimes u_{i+1}^{\mu_{i+1}}\otimes\cdots\otimes  u_{t}^{\mu_{t}}\right)=0
\end{matrix}
 \]
 for $1\leq i\leq t $ and $w\in V$.  
Again we apply the ideas in Section~\ref{sec:multilinear}. Indeed 
\[ u_{1}^{\mu_{1}}\otimes\cdots\otimes u_{i-1}^{\mu_{i-1}}\otimes w\cdot u_{i}^{\mu_{i}-1}\otimes u_{i+1}^{\mu_{i+1}}\otimes\cdots\otimes  u_{t}^{\mu_{t}}\] symmetrizes to $\frac{u_{1}^{\mu_{1}}\cdots u_{t}^{\mu_{t}}}{u_{i}} w$, so
\[\begin{matrix} 
F\left(\frac{u_{1}^{\mu_{1}}\cdots u_{t}^{\mu_{t}}}{u_{i}} w \right) = 0, 
& \text{for all }1\leq i\leq p& \text{for all } w\in V.
\end{matrix}
\]
But since $u_{1}^{\mu_{1}}\cdots u_{t}^{\mu_{t}} = 
v_{1}^{\lambda_{1}}\cdots v_{s}^{\lambda_{s}}$
 and $u_{i}= v_{j}$ for some $i,j$, 
\[ F\left(\frac{u_{1}^{\mu_{1}}\cdots u_{t}^{\mu_{t}}}{u_{i}} w\right) = 
 F\left(\frac{v_{1}^{\lambda_{1}}\cdots v_{s}^{\lambda_{s}}}{v_{j}} w\right) = 0. \]
This holds for all $w\in V$, and these are the conditions that $F$ be tangent to $\Chow_{\lambda}(\PP V)$ at $[u_{1}^{\mu_{1}}\cdots u_{t}^{\mu_{t}}]=[v_{1}^{\lambda_{1}}\cdots v_{s}^{\lambda_{s}}]$ so we are done.
\end{proof}

\begin{proof}[Proof of Theorem~\ref{thm:main theorem}]
The proof of the theorem is now just the combination of Lemmas~\ref{lem:part 1} and~\ref{lem:part 2}.
Lemma~\ref{lem:part 1} showed that \[
 \bigcup_{\lambda\prec \mu} \Chow_{\lambda}\left(\PP V\right)^{\vee} \subset
\Seg_{\mu}\left(\PP V^{ \times t}\right)^{\vee}\cap \PP \left(S^{d}V^{*}\right) 
,\]
and moreover that each $\Chow_{\lambda}\left(\PP V\right)^{\vee}$ that is a hypersurface occurs with multiplicity $M_{\lambda,\mu}$.

For the other direction, apply Lemma~\ref{lem:part 2}. Suppose $F\in \Seg_{\mu}\left(\PP V^{ \times t}\right)^{\vee}\cap \PP \left(S^{d}V^{*}\right)$.  Then $F$ is a symmetric hyperplane, and moreover, $F$ must be tangent to $\Seg_{\mu}\left(\PP V^{ \times t}\right)$
 in some point $[v_{1}^{\otimes \mu_{1}}\otimes\cdots\otimes v_{t}^{\otimes \mu_{t}}]$, and tangent to 
 $\Chow_{\lambda}\left(\PP V\right)$ for every $\lambda$ such that 
 $v_{1}^{ \mu_{1}}\cdots  v_{t}^{ \mu_{t}} \in \Chow_{\lambda}(\PP V)$ and more specifically for every $\lambda \prec \mu$.  This means that $F \in \Chow_{\lambda}\left(\PP V\right)^{\vee}$ for such $\lambda$, and therefore\[
\Seg_{\mu}\left(\PP V^{ \times t}\right)^{\vee}\cap \PP \left(S^{d}V^{*}\right) \subset \bigcup_{\lambda\prec \mu} \Chow_{\lambda}\left(\PP V\right)^{\vee}
.\qedhere\]
\end{proof}

Theorem~\ref{thm:hyp} is a specific case of Theorem~\ref{thm:main theorem}, we only need to note that $M_{\lambda,\lambda} = \binom{d}{\lambda}$ is the binomial coefficient (see Section~\ref{sec:combinat}).
Corollary~\ref{cor:degree} also follows from Theorem~\ref{thm:main theorem}.  This is because in Section~\ref{sec:combinat} we also showed that the multiplicities $M_{\lambda,\mu}$ can be both computed and organized in a lower triangular matrix.  Using the generating function for $D_{\mu}=\deg(\Seg_{\mu}(\PP V^{ \times t}))$, found \cite[Theorem~3.1, Proposition~ 3.2]{GKZ92} or also \cite[page 454]{GKZ}, we can compute the vector of degrees $(D_{\mu})_{\mu}$. Therefore we can solve the linear system $(D_{\mu})_{\mu}=(M_{\lambda,\mu})_{\lambda,\mu} (d_{\lambda})_{\lambda}$, where $d_{\lambda}$ denotes the degree of $\Chow_{\lambda}(\PP V^{ \times s})$. For this we use Proposition~\ref{mark} to compute $M_{\lambda,\mu}$ efficiently.  See the appendix for a few examples.

\subsection{A degree formula in the binary case}
\begin{theorem}
Suppose $V = \CC^{2}$. Let $\lambda = (1^{m_{1}},2^{m_{2}}, \ldots , p^{m_{p}})$, with $m= \sum_{i} m_{i}$ and suppose $m_{1}=0$.
  The degree of $\Chow_{\lambda}(\PP^{1})^{\vee}$
  is
  \begin{equation}\label{eq:bindeg}
(m+1)
\binom{m}{m_{2}, \ldots , m_{p}}1^{m_{2}}2^{m_{3}}\cdots (p-1)^{m_{p}} 
\end{equation}
\end{theorem}

\begin{proof}
Consider the hyperdeterminant of format $\kappa = (k_{1},\dots,k_{r})$, which by \cite[Theorem~XIV.2.5]{GKZ} has degree:
\[
\deg(\Seg(\PP^{k_{1}}\times \dots \times \PP^{k_{r}})^{\vee}) = 
\sum_{\lambda} 
(m_{2}+m_{2}+\dots+m_{p} +1)! \cdot d_{\kappa,\lambda}
\cdot \prod_{i=2}^{p}\frac{(i-1)^{m_{i}}}{m_{i}!}
,\]
where  the sum is over $\lambda =(1^{m_{1}},2^{m_{2}},\dots,p^{m_{p}} )$ with $m_{1}=0$, $\kappa = (k_{1},\dots,k_{r})$, and $d_{\kappa,\lambda}$ is the Gale-Ryser matrix (whose $\kappa, \lambda$ entry corresponds to the number of $0$-$1$ matrices with row sums $\kappa$ and column sums $\lambda$).

By the Fundamental Theorem of Algebra, 
\[
\Seg(\PP^{k_{1}}\times \dots \times \PP^{k_{r}}) \cong 
\Seg_{\kappa}(\PP^{1}\times \dots \times \PP^{1})
,\]
so 
$\deg(\Seg(\PP^{k_{1}}\times \dots \times \PP^{k_{r}})^{\vee}) = \deg(\Delta_{\kappa,1})$.

Let $m=\sum_{i}m_{i}$, $m_{1}=0$ 
and compare to the formula given by Corollary~\ref{cor:degree}:
\begin{equation}\label{bla}
\deg(\Delta_{ \kappa, 1}) = 
\sum_{\lambda \prec \kappa, \lambda\vdash d, \kappa \vdash d} 
\deg(\Chow_{\lambda}\PP^{1})^{\vee}\cdot M_{\lambda,\kappa}
\end{equation}
\[=
\sum_{\lambda} 
(m_{2}+m_{2}+\dots+m_{p} +1)! \cdot d_{\kappa,\lambda}
\cdot \prod_{i=2}^{p}\frac{(i-1)^{m_{i}}}{m_{i}!}
.\]
\[=
\sum_{\lambda } 
(m+1)\binom{m}{m_{2},\dots,m_{p}} 
\cdot \prod_{i=2}^{p}(i-1)^{m_{i}}\cdot d_{\kappa,\lambda}
.\]

The claim is proved noting that $d_{(1^{r}),\lambda}=M_{\lambda,(1^{r})}$ (see Proposition~\ref{prop:M properties}), and that 
\[(m+1)\binom{m}{m_{2},\dots,m_{p}} 
\cdot \prod_{i=2}^{p}(i-1)^{m_{i}} = \deg(\Chow_{\lambda}\PP^{1})^{\vee}\] provides a solution to the system of equations given by varying $\kappa$ in \eqref{bla}, but the solution is unique by Corollary~\ref{cor:degree}.
\end{proof}
While $d_{(1^{r}),\lambda}=M_{\lambda,(1^{r})}$, in general $d_{\kappa,\lambda}\geq M_{\lambda,\kappa}$, (the number of partitions that dominate a given partition is more than the number of partitions that refine it) so expressing $\deg(\Delta_{\kappa,1})$ as an expression involving the $d_{\kappa,\lambda}$ instead of the $M_{\lambda,\kappa}$  will involve a possibly different combination of degrees of duals of Chow varieties. 

The degree formula in the binary case is the same as that of a resultant $R_{f_{0},\dots,f_{m}}$ of type $(m_{2},\dots,m_{p};1,2,\dots,p-1)$.
So another proof strategy would be to find a way to relate this dual variety to a resultant whose degree is equal to the degree we have written above.  This can be done, and essentially only relies on the Fundamental Theorem of Algebra, but for brevity we omit it.

\subsection{A generating function for the degree of dual of a Chow variety}\label{sec:computation}
Utilizing the expression in Proposition~\ref{mark} relating power-sum symmetric functions and monomial symmetric functions, we can improve Theorem~\ref{thm:main theorem}, and provide a generating function for the degree of the duals of the Chow varieties (when they are hypersurfaces).

\begin{theorem}\label{thm:degs}
Suppose $\dim V \geq 2$.  Let $d_{\lambda}$ denote $ \deg(\Chow_{\lambda}\left(\PP V\right)^{\vee}) $ when it is a hypersurface and $0$ otherwise.  Let $\Delta_{ \mu, n}$ denote the equation of the hypersurface $\Seg_{\mu}(\PP V^{\times t})^{\vee}$.
The degrees $d_{\lambda}$ are computed by
\[
\sum_{\mu}\deg(\Delta_{ \mu, n})m_{\mu}(x) = \sum_{\lambda } d_{\lambda}p_{\lambda}(x)
,\]
where $m_{\mu}$ and $p_{\lambda}$ are respectively the monomial and power sum symmetric functions.
\end{theorem}

\begin{proof}
By Corollary~\ref{cor:degree} we have 
\[
\deg(\Delta_{ \mu, n}) = \sum_{\lambda \prec \mu} d_{\lambda}M_{\lambda,\mu}
.\]
Multiply by the monomial symmetric functions $m_{\mu}(x)$ on both sides and sum over all partitions $\mu$, to get
\[
\sum_{\mu\vdash d}\deg(\Delta_{ \mu, n})m_{\mu}(x) = \sum_{\mu\vdash d}\sum_{\lambda\vdash d} d_{\lambda}M_{\lambda,\mu}m_{\mu}(x)
.\]
Change the order of summation and apply Proposition~\ref{mark}:
\[
\sum_{\mu\vdash d}\deg(\Delta_{ \mu, n})m_{\mu}(x) = \sum_{\lambda\vdash d}d_{\lambda}\sum_{\mu\vdash d} M_{\lambda,\mu}m_{\mu}(x) = 
\sum_{\lambda\vdash d}d_{\lambda}p_{\lambda} (x)
.\qedhere\]
\end{proof}
Theorem~\ref{thm:degs} can also provide an alternate proof of Theorem~\ref{thm:hypersurface} because it predicts $d_{\lambda}=0$ precisely when $\Chow_{\lambda}(\PP V)^{\vee}$ is not a hypersurface.

Theorem~\ref{thm:degs} gives an effective way to compute the degrees of the duals of the Chow varieties because we have a generating function for the degree of the $\mu$-discriminant given by \cite[Theorem~XIII.2.4 p.441]{GKZ}.
Combining the GKZ generating function with Theorem~\ref{thm:degs}, we can give a generating function for the degrees of the duals of Chow varieties.

\begin{prop}\label{prop:gf}
Let $d_{\lambda}$ be the degree of $\Chow_{\lambda}(\PP V^{\times s})^{\vee}$ and extend $d_{\lambda}$ to $d_{\kappa,\lambda} = \begin{cases}d_{\lambda} &\mbox{ if }\kappa = (n^{t}) \\ 0 & \mbox{ else} \end{cases}$. 
Then 
\[
\sum_{\kappa}\sum_{\lambda\vdash d}d_{\kappa,\lambda}p_{\lambda}(x) z^{\kappa}
= 
\sum_{\mu\vdash d}
\frac
{1}
{\left[\prod_{i}(1+z_{i}) - \sum_{j}\mu_{j}z_{j}\prod_{i\neq j}(1+z_{i})\right]^{2}}
m_{\mu}(x).
\]
\end{prop}

\begin{proof}
From \cite[Theorem~XIII.2.4 p.441]{GKZ} we have
\[
\sum_{\kappa}N(\kappa;\mu)z^{\kappa}
=\frac
{1}
{\left[\prod_{i}(1+z_{i}) - \sum_{j}\mu_{j}z_{j}\prod_{i\neq j}(1+z_{i})\right]^{2}}
,\]
where $N(\kappa;\mu)$ is the degree of $\Seg_{\mu}(\PP^{k_{1}}\times\dots\times\PP^{k_{t}})^{\vee}$ and $\kappa \in \ZZ^{t}_{>0}$. Since we only care about the hyperdeterminants where the dimensions $k_{i}$ are all equal, i.e.
\[
N(n,\dots,n;\mu_{1},\dots,\mu_{t}) = \deg(\Delta_{\mu,n})
,\]
we consider the coefficient of $z_{1}^{n}\dots z_{t}^{n}$ on both sides. We denote by $\left\langle z_{1}^{n}\dots z_{t}^{n} \right\rangle$ the operation ``take the coefficient of $z_{1}^{n}\dots z_{t}^{n}$''.

Multiply by $m_{\mu}(x)$, sum over all $\mu\vdash d$, and apply Theorem~\ref{thm:degs}
\[
\sum_{\lambda\vdash d}d_{\kappa,\lambda}p_{\lambda}(x)=
\sum_{\mu\vdash d}
N(n,\dots,n;\mu_{1},\dots,\mu_{t}) m_{\mu}(x)
\]
\[=\left\langle z_{1}^{n}\dots z_{t}^{n}\right\rangle\sum_{\mu\vdash d}
\frac
{1}
{\left[\prod_{i}(1+z_{i}) - \sum_{j}\mu_{j}z_{j}\prod_{i\neq j}(1+z_{i})\right]^{2}}
m_{\mu}(x)
,\] and this implies the result.
\end{proof}

\subsection{Examples}
We can use Proposition~\ref{prop:gf} to compute the degrees of the duals to the Chow varieties explicitly.  We found it convenient to separately apply
Theorem~\ref{thm:degs} and Proposition~\ref{mark} to do the same computation.  We have included our Maple code that does this in the ancillary files to the arXiv version of this paper.

Let $d_{\lambda}$ denote the degree of $\Xi_{\lambda}$, and let $D_{\mu}$ denote the degree of $\Delta_{ \mu, n}$.
Consider the case of octic curves, $d=8$ and $n=2$.  Using the GKZ generating function, we find that 
\begin{multline*}
(D_{(8)}, D_{(6,2)}, D_{(5,3)}, D_{(4,4)}, D_{(4,2,2)} , D_{(3,3,2)} , D_{(2,2,2,2)}, D_{(1^{8})} ) 
\\=
(14, 44, 62,68 , 200 , 236, 848, 60032)
.\end{multline*}
The unique solution to $M_{\lambda,\mu}d_{\lambda} =D_{\mu}$ is
\[
(d_{(8)}, d_{(6,2)}, d_{(5,3)}, d_{(4,4)}, d_{(4,2,2)}, d_{(3,3,2)}, d_{(2,2,2,2)}) =
(14,30, 48, 27, 36, 48, 5)
. \]
Notice that $d_{(2,2,2,2)}=5$ is a relic of the fact that $\Chow_{(2,2,2,2)}(\PP^{1})^{\vee}$ is the hypersurface given by the determinant of the $5\times5$  catalecticant (Hankel) matrix of partial derivatives. 
To a binary form $a_{{8},0} x^{8}+a_{{7},1} x^{7} y+a_{{6},{2}} x^{6} y^{2}+a_{{5},{3}} x^{5} y^{3}+a_{{4},{4}} x^{4} y^{4}+a_{{3},{5}} x^{3} y^{5}+a_{{2},{6}} x^{2}
      y^{6}+a_{1,{7}} x y^{7}+a_{0,{8}} y^{8}$ the associated Hankel matrix is
\[ \bgroup\begin{pmatrix}a_{{8},0}&
      a_{{7},1}&
      a_{{6},{2}}&
      a_{{5},{3}}&
      a_{{4},{4}}\\
      a_{{7},1}&
      a_{{6},{2}}&
      a_{{5},{3}}&
      a_{{4},{4}}&
      a_{{3},{5}}\\
      a_{{6},{2}}&
      a_{{5},{3}}&
      a_{{4},{4}}&
      a_{{3},{5}}&
      a_{{2},{6}}\\
      a_{{5},{3}}&
      a_{{4},{4}}&
      a_{{3},{5}}&
      a_{{2},{6}}&
      a_{1,{7}}\\
      a_{{4},{4}}&
      a_{{3},{5}}&
      a_{{2},{6}}&
      a_{1,{7}}&
      a_{0,{8}}\\
      \end{pmatrix}\egroup
,\]
and the determinant of this matrix gives the degree 5 hypersurface associated to the (closure of) forms that are sums of four 8th powers. The analogous feature propagates to all $(\Chow_{(2,2,\dots,2)} \PP^{1})^\vee$.

Proceeding in the same way, in the case $d=4, n=3$ we find
\[(D_{(4)}  , D_{(3,1)} , D_{(2,2)}, D_{(2,1,1)}, D_{(1^{4})})
=
(27, 27, 129, 225, 1269).
\]
and 
the unique solution to $M_{\lambda,\mu}d_{\lambda} =D_{\mu}$ is
\[
( d_{(4)}  , d_{(2,2)} , d_{(2,1,1)},  d_{(1^{4})}) =
( 27, 51, 48, 15)
 .\]

Finally for $d=5, n=4$, we have
%\[\begin{pmatrix}
%1  \\
%1  \\
%1  &1 \\
%1  &1 & 2 \\
%1 & 2  & & 2 \\
%1 & 4  & 6  & 6 & 6\\
%1 & 10 & 20 & 30 & 60 & 120 
%\end{pmatrix}
\begin{multline*}
(D_{(5)} , D_{(4,1)} , D_{(3,2)} , D_{(3,1,1)}, D_{(2,2,1)}, D_{(2,1^{3})} , D_{(1^{5})})\\
= (48, 48, 360, 576, 1440, 7128, 68688),\end{multline*}
and the unique solution to $M_{\lambda,\mu}d_{\lambda} =D_{\mu}$ is
\[
( d_{(5)}  , d_{(3,2)} , d_{(3,1,1)}, d_{(2,2,1)}, d_{(2,1^{3})}, d_{(1^{5})}) =
( 48, 312, 108, 384, 480, 192)
 .\]
To produce more examples, we are only limited by our ability to handle more coefficients of larger power series.

\section*{Acknowledgements}  We would like to thank G. Ottaviani for suggesting this work and for the many useful discussions and encouragement, as well as for his hospitality while the author was a post-doc under his supervision at the University of Florence. Part of the research for this work was done while visiting the Mittag-Leffler Institute. We would like to thank the participants for useful discussions and for the stimulating environment made by their passion for mathematics. J.M. Landsberg, C. Peterson, K. Ranestad, J. Buczynski, and F. Block also provided particularly  useful discussions.  We thank Mark Haiman for pointing out Proposition~\ref{mark}.  We also thank the anonymous referee for useful comments.

\bibliographystyle{amsalpha}
\bibliography{hyperdet_bib}

\end{document}